\documentclass[11pt]{amsart}

\pdfoutput=1

\usepackage{lmodern}
\usepackage[T1]{fontenc}
\usepackage[utf8]{inputenc}

\usepackage{mathtools}
\usepackage{amssymb}
\usepackage{amsthm}
\usepackage{graphicx}
\mathtoolsset{showonlyrefs}
\usepackage{enumitem}
\usepackage{hyperref}

\usepackage{tikz}
\usetikzlibrary{calc}
\tikzset{every picture/.style={line width=1pt},every node/.style={line width=1.5pt,circle,fill=none,draw,inner sep=0,minimum size=13pt}}
\definecolor{color1}{rgb}{0.62068966,0.06896552,1}
\definecolor{color2}{rgb}{0,0.5862069,0}
\definecolor{color3}{rgb}{0.75862069,0.20689655,0}
\definecolor{color4}{rgb}{0.03448276,0.51724138,0.72413793}
\definecolor{color5}{rgb}{0.68965517,0.5862069,0}
\definecolor{color6}{rgb}{0.72413793,1,0.10344828}

\newcommand{\appendixspace}{\hspace*{2em}}
\usepackage{caption}
\usepackage{subcaption}
\usepackage{hyperref} 
\usepackage{bookmark}

\def\E{\mathbb{E}}

\def\E{\mathbb{E}}

\def\P{\mathbb{P}}

\def\eps{\epsilon}

\def\1{\mathbf{1}}

\def\lam {\lambda}
\def\Lam{\Lambda}
\def\tce{t_c + \eps}
\def\tce2{t_c + \frac{\eps}{2}}

\newtheorem*{theorem*}{Theorem}
\newtheorem{theorem}{Theorem}
\newtheorem{lemma}{Lemma}
\newtheorem{cor}{Corollary}

\newtheorem{prop}{Proposition}
\newtheorem*{prop*}{Proposition}
\newtheorem{conj}{Conjecture}
\newtheorem{claim}{Claim}

\begin{document}
\title[Extremes of the internal energy of the Potts model]{Extremes of the internal energy of the Potts model on cubic graphs}

\author{Ewan Davies}
\author{Matthew Jenssen}
\author{Will Perkins}
\author{Barnaby Roberts}
\address{London School of Economics}
\email{\{e.s.davies,m.o.jenssen,b.j.roberts\}@lse.ac.uk }
\address{University of Birmingham}
\email{math@willperkins.org}
\date{\today}
\keywords{Potts model, partition function, graph colorings, graph homomorphims, Ising model}

\begin{abstract}
We prove tight upper and lower bounds on the internal energy per particle (expected number of monochromatic edges per vertex) in the anti-ferromagnetic Potts model on cubic graphs at every temperature and for all $q \ge 2$. This immediately implies corresponding tight bounds on the anti-ferromagnetic Potts partition function.  

Taking the zero-temperature limit gives new results in extremal combinatorics:  the number of $q$-colorings of a $3$-regular graph, for any $q \ge 2$, is maximized by a union of $K_{3,3}$'s. 
This proves the $d=3$ case of a conjecture of Galvin and Tetali. 
\end{abstract}

\maketitle
\thispagestyle{empty}

\section{The Ising and Potts models}

The Potts model is a probabilistic model of interacting \emph{spins} on a graph. Here, we use the term \emph{color} instead of spin to highlight a connection to extremal combinatorics which we cover in Section~\ref{sec:colorings}. 
Let $G=(V,E)$ be a graph and $\sigma \in [q]^{V(G)}$ a coloring (not necessarily proper) of the vertices of $G$ with $q$ possible colors.  Let $m(\sigma)$ denote the number of monochromatic edges of $G$ under $\sigma$. Then the $q$-color Potts model partition function is:
\begin{align*}
Z_G^q(\beta)&= \sum_{\sigma \in [q]^{V(G)} } e^{ - \beta m(\sigma) } .
\end{align*}
The parameter $\beta$ is the \emph{inverse temperature} and  the model is \emph{antiferromagnetic} if $\beta >0$ and \emph{ferromagnetic} if $\beta <0$.  (For general statistical physics terminology, we refer the reader to Chapter 2 of~\cite{mezard2009information}, for example).  The Potts partition function also plays an important role in graph theory as it is an evaluation of the Tutte polynomial of $G$.

The Potts model~\cite{potts1952some} (with no external field) is a random $q$-coloring $\sigma$ of $V(G)$ chosen according to the distribution
\[ \sigma \mapsto \frac{ e^{-\beta m(\sigma)}   }{  Z_G^q(\beta)}.\]
Thus the partition function $Z_G^q(\beta)$ is the normalizing constant that makes this a probability distribution. 
When $\beta$ is positive, the model prefers colorings with fewer monochromatic edges, and the effect is intensified as $\beta$ gets large.
When $\beta$ is negative (the ferromagnetic Potts model), the distribution is biased towards colorings with more monochromatic edges.

The Potts model generalizes the \emph{Ising model}~\cite{ising1925beitrag} (the case $q=2$).  See~\cite{wu1982potts} for a survey of the Potts model. 

The negative of the logarithmic derivative of $Z_G^q$ with respect to $\beta$  gives the expected number of monochromatic edges, or the \emph{internal energy} of the model.  If we scale by the number of vertices, this gives the \emph{internal energy per particle}, $U^q_G(\beta)$:

\begin{align*}
U^q_G(\beta) &:= -\frac{1}{|V(G)|} \frac{d}{d \beta}  \big(\log Z_G^q(\beta)\big)\\
&= -\frac{1}{|V(G)|} \frac{  1}{ Z_G^q(\beta)}\frac{ d}{d \beta}\big( Z_G^q(\beta)\big)\\
 &= \frac{1}{|V(G)|} \frac{  \sum_{\sigma \in [q]^{V(G)}} m(\sigma )e^{- \beta m(\sigma) } }{ Z_G^q(\beta)} \\
&=  \frac{1}{|V(G)|} \E_\sigma  [m(\sigma)].
\end{align*}

When $\beta=0$ there are no interactions in the model, and  $Z_G^q(0) = q^{|V(G)|}$ for all $G$.  Starting from here, we can integrate the internal energy per particle to obtain the scaled logarithm of the partition function, or the \emph{free energy per particle}, $F^q_G(\beta) := \frac{1}{|V(G)|} \log Z^q_G(\beta)$. 

\begin{align}
\label{eq:integrateFormula}
F^q_G(\beta) = \frac{1}{|V(G)|} \log  Z^q_G(\beta) &= \log q -  \int_{0}^\beta U^q_G(t) \, dt .
\end{align}

In this paper we  derive  tight bounds on $U^q_G(\beta)$ for cubic ($3$-regular) graphs in the anti-ferromagnetic ($\beta>0$) regime. From \eqref{eq:integrateFormula} these bounds immediately imply corresponding tight bounds on the free energy of the Ising and Potts models and hence the respective partition functions.  We determine, for every $q$, the maximum and minimum of both the internal energy and the free energy per particle as well as the family of graphs that achieve these bounds. 

Recall that $K_{d,d}$ is the complete $d$-regular bipartite graph on $2d$ vertices and $K_{d+1}$ is the complete graph on $d+1$ vertices.

\begin{theorem}
\label{thm:d3Anti}
For any cubic graph $G$, any $q \ge 2$, and any $\beta >0$,
\[ U^q_{K_{3,3}}(\beta) \le U^q_G(\beta) \le U^q_{K_{4}}(\beta).\]
Furthermore, the respective equalities hold if and only if $G$ is a union of $K_{3,3}$'s or a union of $K_{4}$'s.  As a corollary via \eqref{eq:integrateFormula}, we have
\[ F^q_{K_4}(\beta) \le F^q_G(\beta) \le F^q_{K_{3,3}}(\beta).\]
\end{theorem}

We conjecture that these bounds extend to higher regularity $d$ (the case $d=2$ is simply a calculation, see Section~\ref{sec:d2q2}).
\begin{conj}
\label{conj:internal}
For any $d$-regular graph $G$, any $q \ge 2$, and any $\beta >0$,
\[ U^q_{K_{d,d}}(\beta) \le U^q_G(\beta) \le U^q_{K_{d+1}}(\beta),\]
and in particular, 
\[ F^q_{K_{d+1}}(\beta) \le F^q_G(\beta) \le F^q_{K_{d,d}}(\beta).\]
\end{conj}

If we restrict ourselves to bipartite regular graphs, then Galvin (building on \cite{galvin2004weighted,kahn2001entropy}) proved that the maximizer of the free energy is $K_{d,d}$.   
\begin{theorem}[Galvin~\cite{galvin2006bounding}]
\label{thm:Galvin}
For any $d$-regular bipartite $G$, any $\beta$ and any $q \ge 2$,
\[F^q_G (\beta ) \le F^q_{K_{d,d}}(\beta) .\]
\end{theorem}
For lower bounds on the ferromagnetic Ising and Potts free energy over regular, bipartite graphs, where the infimum is the `Bethe free energy', or the free energy of the infinite $d$-regular tree, see Ruozzi~\cite{ruozzi2012bethe,ruozzi2013beyond} and Csikv{\'a}ri~\cite{csikvari2016extremal}.

A bound such as Theorem~\ref{thm:Galvin} was known without the bipartite restriction in one case previously: in the anti-ferromagnetic Ising model. An extension of Galvin's result by Zhao~\cite{zhao2011bipartite} using the `bipartite swapping trick' gives the following.  
\begin{theorem}[Zhao~\cite{zhao2011bipartite}]
For any $d$-regular graph $G$, $\beta>0$, and $q=2$ (the Ising model),
\[ F^q_G (\beta ) \le F^q_{K_{d,d}}(\beta) .\]
\end{theorem} 

Zhao's method does not work for $q\ge 3$, and in the ferromagnetic phase Galvin's result cannot be extended to all $G$; $K_{d,d}$ is not the maximizer. 
The clique $K_{d+1}$ has a higher free energy for any $d$ when $\beta<0$.  It is natural to conjecture that $K_{d+1}$ is in fact extremal in this case, and also that Galvin's result can be extended to triangle-free graphs.

\begin{conj}\label{conj:ferro}
For any $d$-regular $G$, any $q \ge 2$, and any $\beta<0$,
\[  U^q_G(\beta) \le U^q_{K_{d+1}}(\beta),  \]
and in particular,
\[  F^q_G(\beta) \le F^q_{K_{d+1}}(\beta). \]
Moreover, if in addition $G$ is triangle-free, then for any $\beta$
\[  U^q_G(\beta) \le U^q_{K_{d,d}}(\beta),  \]
and in particular,
\[  F^q_G(\beta) \le F^q_{K_{d,d}}(\beta). \]
\end{conj}

The main contribution of this paper is to prove Theorem~\ref{thm:d3Anti}. We do so by considering the following experiment.  Fix $q$, $\beta$ and a $d$-regular graph $G$.  Choose a vertex $v$ uniformly from $V(G)$ and independently sample a coloring $\sigma \in [q]^{V(G)}$ from the Potts model.  Now for each neighbor $u$ of $v$, record the number of its `external' neighbors (neighbors outside $v \cup N(v)$) receiving each color; record also any edges within $N(v)$ (see Figure~\ref{fig:lvs} for examples). This gives a \emph{local view} of $\sigma$ from $v$.  Note that although we have sampled a coloring $\sigma$ of the whole graph, the colors of $v$ and its neighbors do not form part of the local view.  In fact it is best to think of these colors as having not been revealed.  An important part of the method is that, conditioned on the local view, the distribution of colorings of $v$ and its neighbors can be determined.  For fixed $d$ and $q$ there are only a finite number of possible local views; call this set of local views $\mathcal L_{d,q}$.  Each $d$-regular graph $G$ and inverse temperature $\beta$ induces a probability distribution on $\mathcal L_{d,q}$.

Not all probability distributions on $\mathcal L_{d,q}$ can arise from a graph; there are certain consistency conditions that must hold.  For example, the expected number of monochromatic edges incident to $v$ must equal the expected number of monochromatic edges incident to a uniformly chosen neighbor of $v$. 
Moreover, we can compute both of these expectations given a probability distribution on $\mathcal L_{d,q}$; in fact they are both linear functions of the probabilities.  
For $d=3$ this constraint is sufficient, but for larger $d$ more are required. 
Another family of consistency conditions are that for every multiset $S$ of size $d$ from $q$ colors, the probability $N(v)$ is colored by $S$ must be the same as the probability $N(u)$ is colored by $S$ for a uniformly chosen neighbor $u$ of $v$.  
Finally, the quantity we wish to optimize, $U_G^q(\beta)$ is also a linear function of the probabilities in the distribution on $\mathcal L_{d,q}$.

So instead of maximizing or minimizing $U_G^q(\beta)$ over all $d$-regular graphs, we relax the problem and instead optimize over all probability distributions on $\mathcal L_{d,q}$ that satisfy the above consistency conditions. This is simply a linear program over $|\mathcal L_{d,q}|$ variables.  For some values of $d$, $q$ and $\beta$ we know this linear program is not tight although we conjecture it to be tight whenever $q \geq d+1 \geq 3$ and $\beta >0$.

 This method builds on previous work on independent sets and matchings~\cite{davies2015independent,davies2016average,perarnau2016counting} and the Widom-Rowlinson model~\cite{cohen2017widom}, but here we generalize the previous approach in two ways: 1) we deal with $q$-spin models instead of $2$-spin models; 2) we deal with soft and hard constraints instead of just hard constraints. The family of linear programs in~\cite{davies2015independent} for matchings was an infinite family of LP's indexed by two parameters - the vertex degree $d$ and a fugacity parameter $\lam >0$ - and the entire family could be solved analytically with a single proof via LP duality.   Here the situation is worse: we have an infinite family of LP's indexed by $d$, $q$, $\beta$. Moreover, while the number of constraints for the matching LP grew linearly in $d$, here the number of constraints needed can grow like the integer partition number of $d$. In this paper we solve the program for $d=3$ where there are 35 variables.

In Section~\ref{sec:proof} we solve the LP (both the maximization and minimization problem) for $d=3$, $q\ge 2$, and all $\beta>0$, and we solve it in a somewhat mechanical way that does not reveal much about generalizations to higher $d$. 

Nevertheless, we suspect that the LP is tight for a much wider set of parameters, including for  $q \ge d+1$ and $\beta >0$.  For some other parameter values the constraints described above are not enough to solve the internal energy minimization problem for all $d,q$, $\beta>0$.  It is easy to find values of $\beta$ so that if $d \ge 4$ and $q\le d$, the minimizer of the LP is smaller than $U_{K_{d,d}}^q(\beta)$.  Two challenges for future work are:
\begin{enumerate}[label=(\textup{\roman*})]
\item Solve the infinite family of LP's with $d\ge 4, q \ge d+1$, $\beta>0$ that we conjecture is tight. 
\item Find additional consistency conditions (constraints) that can further tighten the LP for $q \le d$.
\end{enumerate}
It is worth comparing this method to the entropy method used in~\cite{galvin2006bounding,galvin2004weighted,kahn2001entropy,zhao2011bipartite}.  The entropy method has the great virtue of generality - the theorems encompass many models all at once.  The method we employ here has the virtue that the results are stronger---we show optimality on the level of the internal energy, which is both a physical observable and the derivative of the free energy---and that it can be used to prove tight bounds in a wide variety of situations, albeit with much model-specific work required. 

In another direction, extremal bounds in the Potts model on graphs of maximum degree $d$ are given by Sokal~\cite{sokal2001bounds}, proving that all zeros of the chromatic polynomial, $P_G(q)$, have modulus bounded by $7.963907 d$. The complete graph $K_{d+1}$ shows that a linear bound in $d$ is best possible. Given our results here, it would be interesting to find the sharp constant in the upper bound; that is, over all $d$-regular graphs $G$ what is the supremum of $R(G)$ where 
\begin{align*}
R(G) &= \max \{ |r|: r \in \mathbb C, P_G(r) =0 \} .
\end{align*}
In particular, for $d\ge 4$, do we have $R(G) \le R(K_{d,d})$? And for $d=3$ do we have $R(G) \le R(K_{4})=3$? For more, see the discussion in Section $9$ of~\cite{sokal2005multivariate}.  

\section{Maximizing the number of \texorpdfstring{$q$}{q}-colorings of \texorpdfstring{$d$}{d}-regular graphs}\label{sec:colorings}

If we take $\beta \to \infty$ in the Potts model, we bias more and more against monochromatic edges, and thus if a \emph{proper} $q$-coloring of $G$ exists, the `zero-temperature' anti-ferromagnetic Potts model is simply the uniform distribution over proper $q$-colorings of $G$. The limit of the  partition function $\lim_{\beta \to \infty} Z_G^q(\beta)= C_q(G)$, the number of proper $q$-colorings of $G$.  Maximizing $C_q(G)$ over different families of graphs  has been the study of much work in extremal combinatorics.  Linial~\cite{linial1983legal} asked which graph on $n$ vertices with $m$ edges maximizes $C_q(G)$.  After a series of bounds by Labeznik and coauthors~\cite{lazebnik1989greatest,lazebnik1990new,lazebnik2007maximum}, Loh, Pikhurko, and Sudakov~\cite{loh2010maximizing} gave a complete answer to this question for a wide range of parameters $q$, $n$, $m$, using the Szemer\'edi's Regularity Lemma to reduce the maximization problem over graphs to a quadratic program in $2^q-1$ variables. 

A similar question in a very different setting is to ask which $d$-regular, $n$-vertex graph maximizes the number of $q$-colorings; or, given that $C_q$ is multiplicative when taking disjoint unions of graphs, which $d$-regular graph maximizes $\frac{1}{|V(G)|} \log C_q(G)$? Although neither question specifies the sparsity of the graph, one can think of Linial's question as a question about dense graphs and this question as one about sparse graphs (and the techniques of~\cite{loh2010maximizing} and this paper reflect this: the Regularity Lemma primarily concerns dense graphs, while statistical mechanics is primarily concerned with sparse, regular graphs).

  For regular graphs,  Galvin and Tetali~\cite{galvin2004weighted} conjectured that $K_{d,d}$ maximizes the normalized number of $q$-colorings over all $d$-regular graphs.
\begin{conj}[Galvin--Tetali~\cite{galvin2004weighted}]
\label{conj:colorings}
For any $q \ge 2$, $d \ge1$, and all $d$-regular graphs $G$,
\begin{equation}
\label{eq:ColorMax}
  C_q(G)^{1/|V(G)|} \le C_q(K_{d,d})^{1/(2d)}.
  \end{equation}
\end{conj}
In the same paper they prove that \eqref{eq:ColorMax} holds for all $d$-regular, \emph{bipartite} $G$. In the language of graph homomorphisms, $C_q(G)$ counts the number of homomorphisms from $G$ into $K_q$, and their results holds for the number of homomorphisms of a $d$-regular bipartite $G$ in to \emph{any} target graph $H$.

Before this work, Conjecture~\ref{conj:colorings} was not known for any pair $(q,d)$ apart from the trivial cases $d=1$, $d=2$, and $q=2$ (see Section~\ref{sec:d2q2}). However, significant partial progress was made in addition to the bipartite case.  Employing the bipartite swapping trick, Zhao~\cite{zhao2011bipartite} showed that for $q \ge (2n)^{2n-2} $, the bipartiteness restriction could be removed for graphs on $n$ vertices.  Galvin~\cite{galvin2013maximizing} then reduced the lower bound on $q$, showing that $q > 2 \binom{nd/2}{4} $ suffices.  Dependence on $n$ in the number of colors is of course not ideal, as it does not prove Conjecture~\ref{conj:colorings} for any pair $(q,d)$, but it does restrict the class of possible counterexamples.   In another direction of partial progress on Conjecture~\ref{conj:colorings}, Galvin~\cite{galvin2015counting} gave an upper bound on $C_q(G)$ for all $d$-regular $G$, that is tight, asymptotically in $d$, on a logarithmic scale. 

The following lemma relates bounds on internal energy per particle to bounds on the number of $q$-colorings. 

\begin{lemma}
\label{lem:betatoColorings}
Fix $d,q$.  If for all $d$-regular $G$, and all $\beta >0$, $U^q_G( \beta) \ge U^q_{K_{d,d}}(\beta)$, then \eqref{eq:ColorMax} holds.
\end{lemma}
\begin{proof}
Let $G$ be any $d$-regular graph. If $C_q(G)=0$ then \eqref{eq:ColorMax} clearly holds. Otherwise, we take logarithms and write
\begin{align*}
\frac{1}{|V(G)|} \log C_q(G) &= \lim_{\beta \to +\infty} \frac{1}{|V(G)|} \log Z^q_G(\beta) \\
 &=  \log q -  \int_{0}^\infty U^q_G(\beta) \, d\beta \\
 &\le \log q - \int_{0}^\infty U^q_{K_{d,d}}(\beta) \, d\beta \\
 &= \frac{1}{2d} \log C_q(K_{d,d})\,.\qedhere
\end{align*}

\end{proof}

As a corollary of Theorem~\ref{thm:d3Anti} and Lemma~\ref{lem:betatoColorings}, we prove Conjecture~\ref{conj:colorings} for $d=3 $ and all $q$. 
\begin{cor}\label{cor:K33coloring}
For any $3$-regular graph $G$, and any $q\ge 2$,
\[ C_q(G)^{1/|V(G)|} \le C_q(K_{3,3})^{1/6}\,,\]
with equality if and only if $G$ is a union of $K_{3,3}$'s. 
\end{cor}

We remark that in a similar fashion, Theorem~\ref{thm:d3Anti} gives that $ C_q(G)^{1/|V(G)|} \ge C_q(K_{4})^{1/4}$ for all cubic graphs $G$, but this result was recently proved for all $d$ by Csikv{\'a}ri (see \cite{yufeiSurvey}).
\begin{theorem}[Csikv{\'a}ri]
\label{thm:colMinPC}
For all $d$, all $q\ge 2$, and all $d$-regular $G$,
\[ C_q(G)^{1/|V(G)|} \ge C_q(K_{d+1})^{1/(d+1)}.\]
\end{theorem}

Csikv{\'a}ri and Lin~\cite{csikvari2016sidorenko} also proved that for any $d$-regular, bipartite $G$, \[ C_q(G)^{1/|V(G)|} \ge  q \left( \frac{q-1}{q}  \right)^{d/2}  \, , \]
a result that, for $q$ large enough as function of $d$,  is tight asymptotically for a sequence of bipartite graphs of diverging girth. 

\section{Proof of Theorem~\ref{thm:d3Anti}}\label{sec:proof}

In this section we prove Theorem~\ref{thm:d3Anti} by formulating and solving the LP described in the introduction. 
For brevity we drop the superscripts in notation for partition functions and internal energy of a graph $G$, writing $Z_G$ and $U_G$ for these quantities.

Recall the experiment which defines the local view: draw a coloring $\sigma\in[q]^{V(G)}$ according to the $q$-color Potts model with inverse temperature $\beta$ and independently, uniformly at random choose a vertex $v\in V(G)$. 
The local view consists of the induced subgraph of $G$ on $v \cup N(v)$, together with, for each $u\in N(v)$, the multiset of colors that appears in $N(u) \setminus (\{v\} \cup N(v))$. 
Four examples are pictured in Figure~\ref{fig:lvs}.
As noted in the introduction, our calculations depend only on the \emph{number} of `external neighbors' of vertices $u\in N(v)$ which receive each color, and not the graph structure between these external neighbors. 
For clarity we draw the vertices themselves.  Let $\mathcal C_q$ denote all possible local views for the $q$-color Potts model on cubic graphs.
 
As $q$ grows larger the number of possible local views grows like $q^{d(d-1)}$. However, if we consider equivalence classes of the local views under permutations of the colors (as detailed below), the number of possible local views is bounded in terms of $d$.  This makes the LP finite for any fixed $d$. For a complete list of local views (up to equivalence) for $d=3$ see Appendix~\ref{sec:lvlist}. 

\begin{figure}[tb]
\begin{subfigure}[b]{0.4\textwidth}\centering
\begin{tikzpicture}[rotate=90,xscale=-0.8,yscale=0.8]
	\node (v) at (0:0) {};
	\node (u1) at (60:2) {};
	\node (u2) at (0:1) {};
	\node (u3) at (-60:2) {};
	
	\draw (v) -- (u1);
	\draw (v) -- (u2);
	\draw (v) -- (u3);
	
	\node[draw=color1] (u11) at ($ (u1) + (30:1) $) {1};
	\node[draw=color1] (u12) at ($ (u1) + (-30:1) $) {1};
	\node[draw=color1] (u21) at ($ (u2) + (30:1) $) {1};
	\node[draw=color1] (u22) at ($ (u2) + (-30:1) $) {1};
	\node[draw=color1] (u31) at ($ (u3) + (30:1) $) {1};
	\node[draw=color1] (u32) at ($ (u3) + (-30:1) $) {1};
	
	\draw (u1) -- (u11);
	\draw (u1) -- (u12);
	\draw (u2) -- (u21);
	\draw (u2) -- (u22);
	\draw (u3) -- (u31);
	\draw (u3) -- (u32);
\end{tikzpicture}
\subcaption{Local view $C_1$}\label{fig:lv2}
\end{subfigure}
\begin{subfigure}[b]{0.4\textwidth}\centering
\begin{tikzpicture}[rotate=90,xscale=-0.8,yscale=0.8]
	\node (v) at (0:0) {};
	\node (u1) at (60:2) {};
	\node (u2) at (0:1) {};
	\node (u3) at (-60:2) {};
	
	\draw (v) -- (u1);
	\draw (v) -- (u2);
	\draw (v) -- (u3);
	
	\node[draw=color1] (u11) at ($ (u1) + (30:1) $) {1};
	\node[draw=color2] (u12) at ($ (u1) + (-30:1) $) {2};
	\node[draw=color1] (u21) at ($ (u2) + (30:1) $) {1};
	\node[draw=color2] (u22) at ($ (u2) + (-30:1) $) {2};
	\node[draw=color1] (u31) at ($ (u3) + (30:1) $) {1};
	\node[draw=color2] (u32) at ($ (u3) + (-30:1) $) {2};
	
	\draw (u1) -- (u11);
	\draw (u1) -- (u12);
	\draw (u2) -- (u21);
	\draw (u2) -- (u22);
	\draw (u3) -- (u31);
	\draw (u3) -- (u32);
\end{tikzpicture}
\subcaption{Local view $C_2$}\label{fig:lv3}
\end{subfigure}
\\[\baselineskip]
\begin{subfigure}[b]{0.4\textwidth}\centering
\begin{tikzpicture}[rotate=90,xscale=-0.8,yscale=0.8]
	\node (v) at (0:0) {};
	\node (u1) at (45:1.4142) {};
	\node (u2) at (0:1) {};
	\node (u3) at (1,-1.4142) {};
	
	\draw (v) -- (u1);
	\draw (v) -- (u2);
	\draw (v) -- (u3);
	\draw (u1) -- (u2);
	
	\node[draw=color1] (u11) at ($ (u1) + (0:0.8660) $) {1};
	\node[draw=color2] (u21) at ($ (u2) + (0:0.8660) $) {2};
	\node[draw=color3] (u31) at ($ (u3) + (30:1) $) {3};
	\node[draw=color4] (u32) at ($ (u3) + (-30:1) $) {4};
	
	\draw (u1) -- (u11);
	\draw (u2) -- (u21);
	\draw (u3) -- (u31);
	\draw (u3) -- (u32);
\end{tikzpicture}
\subcaption{Local view with $1$ triangle}
\end{subfigure}
\begin{subfigure}[b]{0.4\textwidth}\centering
\begin{tikzpicture}[rotate=90,xscale=-0.8,yscale=0.8]
	\node (v) at (0:0) {};
	\node (u1) at (30:1.73205) {};
	\node (u2) at (0:1) {};
	\node (u3) at (-30:1.73205) {};
	
	\draw (v) -- (u1);
	\draw (v) -- (u2);
	\draw (v) -- (u3);
	\draw (u1) -- (u2);
	\draw (u2) -- (u3);
	
	\node[draw=color1] (u11) at ($ (u1) + (0:0.8660) $) {1};
	\node[draw=color2] (u31) at ($ (u3) + (0:0.8660) $) {2};
	
	\draw (u1) -- (u11);
	\draw (u3) -- (u31);
\end{tikzpicture}
\subcaption{Local view with $2$ triangles}
\end{subfigure}
\caption{\label{fig:lvs}Example local views. 
Figures \subref{fig:lv2} and~\subref{fig:lv3} show, up to permutations of the colors, the only local views which can arise in $K_{3,3}$.  The colored, numbered vertices are representations of the multiset of colors that appear in the boundary $N(u) \setminus (\{v\} \cup N(v))$ for $u\in N(v)$.}
\end{figure}
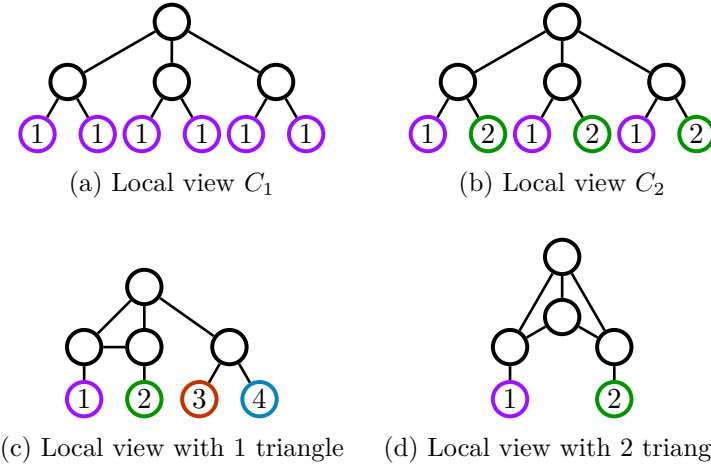

Suppose that the local view $C$ arises from selecting the coloring $\sigma$ and the vertex $v$. 
We refer to the colored `external neighbors' at distance two from $v$ as the \emph{boundary}, and write $V_C$ for the set of uncolored vertices in $C$, so that the set of $q$-colorings of these vertices is $[q]^{V_C}$. 
The coloring $\sigma$ induces a \emph{local coloring} $\chi\in[q]^{V_C}$ that, by the Markov property of the Potts model, is distributed according to the Potts model on $C$. 
For $\chi\in[q]^{V(C)}$, write $m(\chi)$ for the total number of monochromatic edges in $C$ (including any monochromatic edges between $V_C$ and the boundary), and given a vertex $u\in V_C$ write $m_u(\chi)$ for the number of monochromatic edges in $C$ incident to $u$. 
Then, with the local partition function defined as
\begin{align}
Z_C(q,\beta) &:= \sum_{\chi\in [q]^{V_C}}e^{-\beta m(\chi)}\,,
\end{align}
a local coloring $\chi\in [q]^{V_C}$ is distributed according to 
\[
\chi\mapsto\frac{e^{-\beta m(\chi)}}{Z_C(q,\beta)}\,.
\]

This fact means that we can interpret the internal energy per particle as an expectation over the random local view $C$ and local coloring $\chi$. 
Each edge of $G$ is incident to exactly two vertices, hence
\begin{align}
U_G(\beta) &= \frac{1}{|V(G)|} \E_\sigma[m(\sigma)]\\
&= \frac{1}{2|V(G)|} \sum_{v\in V(G)}\sum_{u\in N(v)} \P_\sigma(\text{$uv$ monochromatic})\\
&= \frac{1}{2}\E_C\Big[\sum_{u\in N(v)} \P(\text{$uv$ monochromatic}|C)\Big]\\
&= \frac{1}{2}\E_{C,\chi}[m_v(\chi)]\,.
\intertext{%
Moreover, since $G$ is regular, a neighbor of $v$ chosen uniformly at random is distributed uniformly over $V(G)$, giving}
U_G(\beta) &= \frac{1}{2}\E_{C,\chi}\Big[\frac{1}{3}\sum_{u\in N(v)} m_u(\chi)\Big]\,.
\end{align}
Given these observations, for a local view $C$ we define
\begin{align}
U^v_C &:= \frac{1}{2Z_C}\sum_{\chi\in[q]^{V_C}} m_v(\chi) e^{-\beta m(\chi)}\,,\\
U^N_C &:= \frac{1}{6Z_C}\sum_{\chi\in[q]^{V_C}}\sum_{u\in N(v)} m_u(\chi) e^{-\beta m(\chi)}\,,\\
\shortintertext{so that}
U_G(\beta) &= \E_C[U^v_C] = \E_C[U^N_C]\,,\label{eq:U}
\end{align}
giving us a \emph{constraint} on probability distributions on local views that holds for all distributions arising from graphs.

We can now define the two LP's for the $q$-color Potts model,
\begin{align}
  \{ U^{\min}, U^{\max} \} = \{ \text{Minimum, Maximum} \} \text{ of } &\sum_{C \in \mathcal C_q} p_{C} U^v_C\qquad\text{ subject to }\\
& p_C \ge 0 \;\forall C \in \mathcal C_q\,,\\
& \sum_{C \in \mathcal C_q} p_C = 1\,,\\
& \sum_{C \in \mathcal C_q} p_C (U^v_C-U^N_C) = 0\,,
\end{align}
The final constraint is an example of a \emph{consistency condition} that holds for all distributions on local views which arise from the experiment on a graph defined at the beginning of this section. 

\subsection{Minimizing}

For the minimization problem, the dual program (with variables $\Lam$, $\Delta$) is
\begin{align}
U^{\min}={}&\max\Lambda\qquad\text{ subject to }& \\
& \Lam+\Delta(U^v_C-U^N_C) \le U^v_C\qquad\forall C \in \mathcal C_q\,.
\end{align}
For a given $\beta >0$ and $q \ge 2$, to show that $U^{\min} = U_{K_{3,3}}$ via linear programming duality, we must find $\Delta^*$ so that the assignment $\Lam = U_{K_{3,3}}, \Delta = \Delta^*$ is \emph{feasible} for the dual. That is, 
\begin{equation}
\label{eq:dualFeas}
U_{K_{3,3}}+\Delta^*(U^v_C-U^N_C) \le U^v_C
\end{equation}
for all $C \in \mathcal C_q$. 

In fact it suffices to show \eqref{eq:dualFeas} on a subset of $\mathcal C_q$. We say $C, C' \in \mathcal C_q$ are equivalent if 
\begin{align*}
U^v_{C} \equiv U^v_{C'} \text{ and } U^N_{C} \equiv U^N_{C'}
\end{align*}
as functions of $q$ and $\beta$.  For instance if $C$ is obtained from $C'$ by a permutation of the $q$-colors, then $C$ and $C'$ are equivalent by symmetry.  This equivalence relation partitions $\mathcal C_q$ into equivalence classes. Call this set of equivalence classes $ \mathcal C_q'$. 
We always choose a representative member of the equivalence class that has an initial segment of the colors $[q]$ on its boundary, and write $q_C$ for the total number of colors on the boundary of a local view $C$. 
For $d=3$ and arbitrary $q$ there are $35$ non-isomorphic equivalence classes which we list in Appendix~\ref{sec:lvlist}.

We give the local views that arise in the optimizing graphs names, writing $C_1$ and $C_2$ (see Figure~\ref{fig:lvs}) for representatives of the only two equivalence classes of local views are that can appear with positive probability when $G = K_{3,3}$. 
In $K_4$, the only local view that can arise is isomorphic to $K_4$ itself.

To find a value of $\Delta^*$ we solve the dual constraint corresponding to local view $C_1$ (Figure~\ref{fig:lv2}) to hold with equality when $\Lam  = U_{K_{3,3}}$:
\[
U_{K_{3,3}}+\Delta^*(U^v_{C_1}-U^N_{C_1}) = U^v_{C_1}\,.
\]
Writing $\lam=e^{-\beta}$ (so that $0<\lam<1$), we find 
\begin{align}
Z_{C_1} &= (\lam^3+q-1)^3+(q-1)(\lam^2+\lam+q-2)^3 \,,\\
U^v_{C_1} &= \frac{3}{2Z_{C_1}}\left(\lam^3(1+\lam^3)^2+(q-1)\lam^2(\lam^2+\lam+q-2)^2\right)\\
U^N_{C_1} &= \frac{1}{2Z_{C_1}}\left(3\lam^3(1+\lam^3)^2+(q-1)(\lam+2\lam^2)(\lam^2+\lam+q-2)^2\right)
\shortintertext{and hence}
U^N_{C_1}-U^v_{C_1} &= \frac{1}{2Z_{C_1}}\lam(1-\lam)(q-1)(\lam^2+\lam+q-2)^2\,.
\end{align}

Then  we calculate
\begin{align}
Z_{K_{3,3}} &= q (\lam^3+q-1)^3+3q(q-1)(\lam ^2+\lam +q-2)^3+
\\&\qquad\qquad\qquad\qquad\qquad\qquad q(q-1)(q-2)(3 \lam +q-3)^3\,,\\
U_{K_{3,3}} &= \frac{3q}{2Z_{K_{3,3}}}\bigg(\lam^3(\lam^3+q-1)^2+\lam (q-1)(q-2)(3 \lam +q-3)^2+\\&\qquad\qquad\qquad\qquad (q-1)(2\lam^2+\lam)(\lam^2+\lam+q-2)^2\bigg)\,,
\end{align}
\begin{align}
\label{eq:delstar}
\Delta^* &=-\frac{3q(1-\lam)^2}{2(\lam^2+\lam+q-2)^2 Z_{K_{3,3}}}
\bigg(2 \left(\lambda ^4+9 \lambda ^3+22 \lambda ^2+27 \lambda +13\right) (\lambda -1)^6+
\\&\qquad\qquad 2 \left(\lambda ^2+2 \lambda -1\right) q^5+\left(\lambda ^5+3 \lambda ^4+20 \lambda ^3-41 \lambda +17\right) q^4+
\\&\qquad\qquad\left(8 \lambda ^4+31 \lambda ^3+91 \lambda ^2+47 \lambda -57\right) (\lambda -1)^2 q^3+
\\&\qquad\qquad\left(25 \lambda ^4+82 \lambda ^3+146 \lambda ^2+22 \lambda -95\right) (\lambda -1)^3 q^2+
\\&\qquad\qquad\left(2 \lambda ^5+36 \lambda ^4+87 \lambda ^3+93 \lambda ^2-31 \lambda -79\right) (\lambda -1)^4 q\bigg)\,.
\end{align}

\begin{claim}
\label{claim:slackMin}
For all $q \ge 2$ and all $\beta >0$, the function
\begin{equation}
\label{eq:slack}
SLACK(C) = U^v_C+\Delta^*(U^N_C - U^v_C) -U_{K_{3,3}}
\end{equation}
with $\Delta^*$ given by \eqref{eq:delstar}
is identically $0$ for $C\in \{ C_1, C_2\}$ and strictly positive for all other $C \in \mathcal C_q'$. 
\end{claim}

Claim~\ref{claim:slackMin} immediately proves that $U^q_G(\beta) \ge U^q_{K_{3,3,}}(\beta)$.  To show uniqueness, observe that strict positivity of the slack function implies via complementary slackness that the support of any distribution achieving the optimum must be contained in $\{ C_1, C_2\}$; $K_{3,3}$ is the only connected graph whose distribution satisfies this.
To see this note that, for any other connected cubic graph, there exists a vertex $v$ with two neighbors $u_1$, $u_2$ such that the external neighborhoods of $u_1$ and $u_2$ are distinct. 
Then there exists a coloring such that the external neighbors of $u_1$ are monochromatic, whilst those of $u_2$ are not. 
This means a local view not isomorphic to $C_1$ or $C_2$ appears with positive probability.

In order to prove Claim~\ref{claim:slackMin}, we change variables and multiply the slack by a positive scaling factor, carefully chosen to result in a polynomial with positive coefficients. 
Write $r=q-3$ and $t=e^\beta-1=1/\lam -1$, so that for any $q\ge 3$ and $\beta>0$ we have $r\ge 0$ and $t>0$. 
It then suffices to show that the following scaling of the slack is non-negative:
\begin{align}\label{eq:sslack}
\tilde S_C = \frac{4(1+t)^{17} (r (1+t)^2+t^2+3 t+3)^2}{(3 + r)t^2} Z_{K_{3,3}}Z_C\cdot SLACK(C)\,.
\end{align}
In fact something stronger is true: 
\begin{claim}
\label{minclaim1}
For all $C \in \mathcal C_q'$, $\tilde S_C$ is a bivariate polynomial in $r$ and $t$ with all coefficients positive.  The polynomial is identically $0$ if and only if $C \in \{C_1, C_2\}$.   
\end{claim}
For the case $q=2$ we do something slightly different.
\begin{claim}
\label{minclaim2}
For all $C\in \mathcal C_2'$ (which necessarily use at most $2$ colors on the boundary),  evaluating $\tilde S_C$ at $r=-1$ yields a polynomial in $t$ with positive coefficients. The polynomial is identically $0$ if and only if $C \in \{C_1, C_2\}$.   
\end{claim}

Claims~\ref{minclaim1} and \ref{minclaim2} are proved by simply computing the functions $\tilde S_C$  for each of the $35$ equivalence classes in $\mathcal C_q'$, simplifying and collecting coefficients.

We include in Appendix~\ref{sec:explaincode} an explanation of a computer program used to compute $\tilde S_C$ for each local view $C$. A version of the program in the Sage mathematical language and its output are available as ancillary files in the arXiv submission of this paper. 
Each of the steps the program performs are readily achievable by hand, though the number of steps and the size of the polynomials involved make this unappealing. 

The output shows a list of $\tilde S_C$ for all 35 non-isomorphic local views, demonstrating that it is zero for $C_1$ and $C_2$ and a non-zero polynomial in $r$ and $t$ with non-negative coefficients for all other $C$.
It also shows $\tilde S_C$ evaluated at $r=-1$ for local views $C$ which use at most $2$ colors on the boundary, yielding a non-zero polynomial in $t$ with non-negative coefficients for all such $C$ except $C_1$ and $C_2$, as desired.

\subsection{Maximizing}

To show that $K_4$ is the unique maximizer of the LP is somewhat more straightforward. 
Since the distribution yielding $K_4$ as a local view with probability one is feasible in the LP, it suffices to show that $U^v_{K_4} > U^v_C$ for all $C\neq K_4$.
\begin{claim}
\label{claim:max}
Let 
\begin{align}\label{eq:sDv}
D^v_C = 2 (1+t)^{14} Z_{K_4} Z_C t^{-2}\big(U^v_{K_4} -U^v_C\big)\,.
\end{align}
Then for all $C \in \mathcal C_q$, $D^v_C$ is a polynomial $t=e^\beta-1=1/\lam -1$ and $s=q-\max\{3,q_C\}$ with all positive coefficients, and identically $0$ if and only if $C= K_4$.
\end{claim}

Since local views with $q_C>q$ cannot occur, for $q\ge3$ and $\beta>0$ we have $s\ge0$ and $t\ge0$ and hence Claim~\ref{claim:max} implies $U^v_{K_4} > U^v_C$ for all $C\neq K_4$. 
The quantity $D^v_C$ is listed for all 35 non-isomorphic local views in Appendix~\ref{sec:lvlist}.
Again, for $q=2$ we must do more; for $C \in \mathcal C_2'$  we list $D^v_C$ evaluated at $q=2$, observing that it is a polynomial in $t$ with non-negative coefficients, except for $K_4$ where it is zero. 

As with the computations for $\tilde S_C$, we use a computer to multiply polynomials and obtain $D^v_C$ for each local view, 
see Appendix~\ref{sec:explaincode}.

\section{Extensions to \texorpdfstring{$d \ge 4$}{d at least 4}?}\label{sec:Alld}

How might we extend Theorem~\ref{thm:d3Anti} to graphs of larger degree?  The minimization program defined above in Section~\ref{sec:proof} is not tight in general: we can in fact see that it is underconstrained by comparing the number of constraints ($2$) to the number of equivalence classes of local views in the support of the distribution induced by $K_{d,d}$, which is the partition number of $d-1$ when $q \ge d-1$, and always more than $2$ if $d \ge 4$ and $q \ge 2$.  

There is a large family of constraints that we can add to the program.  Let $\mathcal S_{q,d}$ be the set of all \emph{$q$-partitions} of size $d$; that is, partitions of $d$ into at most $q$ parts which we represent by vectors of length $q$ with non-negative integer entries that sum to $d$, written in non-decreasing order.  Any $q$-coloring $\chi$ of $d$ vertices induces a $q$-partition; for instance if $\chi$ assigns the colors $\{1,4,2,2,1,2\}$, then the $q$-partition $H(\chi) = \{3,2,1,0\} \in \mathcal S_{4,6}$.  Our family of constraints will be that for every $S \in \mathcal S_{q,d}$, the probability that the neighbors of $v$ receive a coloring with $q$-partition $S$ equals the average probability of the same for a neighbor of $v$.

Both of these probabilities can be computed as expectations over the random local view. For a local view $C$ and a $q$-partition $S \in \mathcal S_{q,d}$ we define
\begin{align*}
\gamma^{v,S}_C &:= \frac{1}{Z_C}\sum_{\chi\in[q]^{V_C}} \mathbf 1_{\{H(\chi(N(v)))=S\}} \cdot e^{-\beta m(\chi)}\,,\\
\gamma^{N,S}_C &:= \frac{1}{d} \frac{1}{Z_C}\sum_{\chi\in[q]^{V_C}}\sum_{u\in N(v)}\mathbf 1_{\{H(\chi(N(u)))=S\}}\cdot e^{-\beta m(\chi)}\,.
\end{align*}
Observe that for any graph and any $q$-partition $S$, we must have
\begin{align*}
 \E_C[\gamma^{v,S}_C] &= \E_C[\gamma^{N,S}_C]\, .\label{eq:gamVu}
\end{align*}

Our minimization program becomes
\begin{align}
 U^{\min} =  \text{Minimum of } &\sum_{C} p_{C} U^v_C\qquad\text{ subject to }\\
& p_C \ge 0 \;\forall C\,,\\
& \sum_{C} p_C = 1\,,\\
& \sum_{C} p_C (\gamma^{v,S}_C-\gamma^{N,S}_C) = 0\, \text{ for all } S\in \mathcal S_{q,d}.
\end{align}

This program is at least as strong as the one used in Section~\ref{sec:proof}: the $q$-partition constraints together imply the constraint $ \E_C[U^v_C] = \E_C[U^N_C]$.

We can solve this program for small values of $d$ and fixed $\beta$, which leads us to the following conjecture.
\begin{conj}
\label{conj:Dlp}
The above minimization LP is tight for $d \ge 3, q \ge d+1$ and all $\beta>0$, and shows that 
\[ U^q_{K_{d,d}}(\beta) \le U^q_G(\beta)  \]
for all $d$-regular $G$.
\end{conj}

However we can also find values of $\beta$ for $d \ge 4, q \le d$ so that $U^{\min} < U^q_{K_{d,d}}$, and so we believe that this program is not tight in these cases.

\section{\texorpdfstring{$2$}{2}-regular graphs and other easy cases}\label{sec:d2q2}

Theorem~\ref{thm:d3Anti} shows that $K_{3,3}$ is optimal on the level of internal energy per particle in the Potts model, and by Corollary~\ref{cor:K33coloring} it maximizes $\frac{1}{|V(G)|} \log C_q(G)$ over cubic graphs $G$. 
For arbitrary $d$, in the case $q=2$, the fact that $K_{d,d}$ maximizes $\frac{1}{|V(G)|} \log C_q(G)$ over $d$-regular $G$ follows simply from the observation that $K_{d,d}$ is the smallest bipartite $d$-regular graph. 
Indeed for $q=2$, if $G$ is not bipartite then $C_q(G)=0$ and if $G$ is bipartite $C_q(G) = 2^{\#\text{ connected components of } G}$.

For $d=2$, the only $d$-regular connected graphs are cycles, and there is an explicit formula for the $q$-color Potts partition function of the $n$-cycle.
In the language of statistical physics the $1$-dimensional Potts model (including the $0$-temperature Potts model) is \emph{exactly solvable}:
\begin{align*}
Z_{C_n}^q(\beta) &= (q-1) (e^{-\beta}-1)^{n}+(e^{-\beta}+q-1 )^n.
\end{align*}
One way to obtain this formula is to use the mapping of the Tutte polynomial $T(x,y)$ to the Potts partition function, given in e.g.\ \cite{welsh2000potts}, and then using the formula $T_{C_n}(x,y) = \frac{x^n -x}{x-1} +y$.

Taking the logarithmic derivative gives:
\begin{align}
\label{eq:d2Beta}
U^q_{C_n} (\beta) &= \frac{e^{-\beta}}{e^{-\beta}-1} \cdot\frac{\left(1+\frac{q}{e^{-\beta}-1}\right)^{n-1}+q-1}{\left(1+\frac{q}{e^{-\beta}-1}\right)^{n}+q-1}
\end{align}

\begin{prop}\label{prop:occmono}
If $\beta>0$ then
\begin{align*}
U^q_{C_n} (\beta)&>U^q_{C_{n+2}} (\beta)\text{ for $n\geq3$ odd},\\
U^q_{C_n} (\beta)&<U^q_{C_{n+2}} (\beta)\text{ for $n\geq4$ even}.
\end{align*}
If $\beta<0$ then
\begin{align*}
U^q_{C_n} (\beta)&>U^q_{C_{n+1}} (\beta)\text{ for all $n\geq3$}.
\end{align*}
\end{prop}
\begin{proof}
Let $\beta>0$ and suppose that $n\ge3$ is odd. Let $x:=1+\frac{q}{e^{-\beta}-1}$. By \eqref{eq:d2Beta}, we then have that $U^q_{C_n} (\beta)>U^q_{C_{n+2}}$ if and only if  
\begin{equation}\label{eq:oddmono}
x^n+x^{n+1}>x^{n-1}+x^{n+2}.
\end{equation}
Since $n$ is odd, \eqref{eq:oddmono} holds if and only if $x+x^2>1+x^3$ which holds since $x<-1$. For even $n\geq4$, the proof is the same. 

Suppose now that $\beta<0$ and $n\geq3$. Defining $x$ as before, we have that $x>0$. In this case, the inequality $U^q_{C_n} (\beta)>U^q_{C_{n+1}} (\beta)$ simply reduces to the inequality $x^{n-1}(1-x)^2>0$.  
\end{proof}

Letting $\mathbb{Z}_1$ denote the infinite line we have that 
\begin{equation}
\label{eq:linelimit}
 U^q_{\mathbb{Z}_1}(\beta)= \frac{e^{-\beta}}{e^{-\beta} + q-1} = \lim_{n\to\infty}U^q_{C_n}(\beta).
 \end{equation}
Taking the limit here is justified as the Potts model on $\mathbb Z_1$ is in the Gibbs uniqueness regime for all $q,\beta >0$.
\begin{cor}
Conjectures~\ref{conj:internal}, \ref{conj:ferro}, and \ref{conj:colorings} hold for $d=2$. Moreover, If $\beta>0$ then 
\begin{align*}
U^q_{C_n} (\beta)&> U^q_{\mathbb{Z}_1}(\beta)\text{ for $n\geq3$ odd},\\
U^q_{C_n} (\beta)&< U^q_{\mathbb{Z}_1}(\beta)\text{ for $n\geq4$ even}.
\end{align*}
If $\beta<0$ then
\begin{align*}
U^q_{C_n} (\beta)&>U^q_{\mathbb{Z}_1}(\beta)\text{ for all $n\geq3$}.
\end{align*}
\end{cor}
\begin{proof}
This follows from Proposition \ref{prop:occmono} and \eqref{eq:linelimit}. 
\end{proof}

\section*{Acknowledgements}

We thank Emma Cohen, Prasad Tetali, and Yufei Zhao for enjoyable and enlightening conversations about graph homomorphism problems. 

\bibliography{coloringsBib}
\bibliographystyle{abbrv}

\appendix
\section{List of local views}\label{sec:lvlist}
\input{listOfLVs.tex}

\section{Computing properties of local views}\label{sec:explaincode}

The verification that, for each of the local views shown in Appendix~\ref{sec:lvlist}, the scaled slack~\eqref{eq:sslack} and the scaled difference~\eqref{eq:sDv} are non-negative (and zero where required) was done with the aid of a computer. 
Here we describe some additional considerations required to perform this verification for arbitrary $q$ and $\beta$.

As noted in Section~\ref{sec:proof}, the number of equivalence classes of local views we must consider is bounded independently of $q$, and we only consider representatives of each equivalence class that use an initial segment of colors from $\{1,\dotsc,6\}$ on the boundary. 
In order to compute the partition function and other properties of a local view $C$, one is required to consider the $q^4$ possible local colorings of $V_C$. 
This too can be done in a way that is bounded independently of $q$ by considering equivalence classes of local colorings. 

Let $C$ be a local view (that uses an initial segment of $\{1,\dotsc,6\}$ on the boundary) and recall that $q_C$ is the largest color appearing on the boundary of $C$. 
Then given a local coloring $\chi$ of $V_C$, we can only see at most $q_C+4$ colors. 
After permuting colors not used on the boundary, we may assume that $\chi$ consists only of colors in $[q_C]$ and initial segment of $\{q_C+1,\dotsc,q_C+4\}$ (which may be empty). 
This means we are considering equivalence classes of local colorings and choosing a representative $\tilde\chi$ of each class such that, together with the colors on the boundary, we only ever color $C$ with an initial segment of $[q_C+4]$.

Then for arbitrary $q$ it suffices to consider at most $q_C+4\le10$ colors in the calculations for $Z_C$, $U^v_C$, and $U^N_C$. 
Given the set $Q_C\subset [q_C+4]^{V_C}$ of \emph{representative} local colorings $\tilde\chi$ such that $\tilde\chi$ uses an initial segment (which may be empty) of the colors $\{q_C+1,\dotsc,q_C+4\}$, and writing $\ell$ for the largest color used in $\tilde\chi$, the Potts model on $C$ induces the distribution
\[
\tilde\chi \mapsto 
\begin{cases*}
\frac{e^{-\beta m(\tilde\chi)}}{Z_C}\binom{q-q_C}{\ell-q_C} & if $\ell>q_C$\\
\frac{e^{-\beta m(\tilde\chi)}}{Z_C}& otherwise
\end{cases*}
\]
on $\tilde\chi\in Q_C$.

The consideration of these equivalence classes of local colorings means that for any $\beta$, $q$, and $C$, the quantities $Z_C$, $U^v_C$, and $U^N_C$ may be computed by summing over $Q_C$ whose size is bounded independently of $q$. 
Using this simplification, we used a \emph{SageMath}\footnote{\url{http://www.sagemath.org}} computer program (hosted with the arXiv version of this paper) to compute the scaled slack function $\tilde S_C$ and the scaled difference $D^v_C$ for each of the 35 local views. 
The program can be used to generate a document (included as a supplementary file) containing all the required polynomials so that the reader may verify the proof. 
In addition the program can check itself for non-negative coefficients and print these observations on request.

\end{document}